\documentclass[a4paper,11pt]{article}
\usepackage[latin1]{inputenc}
\usepackage{amsthm}
\usepackage{amsmath,amssymb,amsfonts}
\usepackage[mathscr]{euscript}
\usepackage{enumerate}%,fullpage}
\usepackage{xspace}
\usepackage{url}
\numberwithin{equation}{section}
\newtheorem{thm}{Theorem}[section]

\newtheorem{prop}[thm]{Proposition}
\newtheorem{lm}[thm]{Lemma}
\newtheorem{cor}[thm]{Corollary}
\newtheorem{Rq}[thm]{Remark}

\theoremstyle{definition}

\newtheorem{ex}[thm]{Example}

\theoremstyle{remark}

\theoremstyle{plain}

\DeclareMathAlphabet{\calptmx}{OMS}{ztmcm}{m}{n}

\newcommand{\N}{\mathbb{N}}

\newcommand{\C}{\mathbb{C}}
\newcommand{\R}{\mathbb{R}}

\newcommand{\T}{\mathbb{T}}

\newcommand{\Nc}{\mathcal{N}}
\newcommand{\Cc}{\mathcal{C}}

\newcommand{\Sb}{\mathbb{S}}
\newcommand{\non}{\noindent}
\newcommand{\Hc}{\mathcal{H}}

\author{Bakri Laurent}
\title{ Critical set of eigenfunctions of the Laplacian}
\begin{document}

%\selectlanguage{english}
\date{}
\maketitle
\begin{center}
% \small{Laboratoire de Math\'ematiques\\
% Universit\'e de Bretagne Occidentale\\
% 6 av. Victor Le Gorgeu -- CS 93837\\
% 29238 Brest Cedex\\
% France\\
{E-mail:  Laurent.bakri@gmail.com}
\end{center}
\medskip
 \begin{abstract}
 We give an upper bound for the $(n-1)$-dimensional
 Hausdorff measure of the critical set of eigenfunctions of the Laplacian on compact analytic Riemannian manifolds.
This is the analog  of H. Donnely and C. Fefferman \cite{DF1} result on nodal set of eigenfunctions.
 \end{abstract}

\section{Introduction and statement of the results}
Let $(M,g)$ be a smooth,  compact and connected, $n$-dimensional Riemannian manifold ($n\geq2)$.
 For $u \in\Cc^1(M)$, we set $$\Nc_u=\{x\in M:u(x)=0\}$$  and $$\Cc_u=\{x\in M:\nabla u(x)=0\},$$ the nodal set of $u$ and the critical set respectively.
It is well kown that if $u$ is a non trivial solution of second order linear elliptic equation then all zeros of $u$ are of finite order (\cite{aron},\cite{horm2}),
 and one can prove that the Hausdorff dimension of the nodal set $\Nc_u$ is at most $n-1$ (for example, see \cite{CF} or  \cite{HS} for more precise results).
 When dealing with the eigenfunctions of the Laplacian  : \begin{equation}\label{fp}- \Delta u=\lambda u,
\end{equation} 
S. T. Yau \cite{yau} has conjectured that  $$C_1\sqrt{\lambda}\leq \Hc^{n-1}(\Nc_{ u})\leq C_2\sqrt{\lambda}$$ where
  $\Hc^{n-1}$ denotes the $(n-1)$-dimensional Hausdorff measure and $C_1$, $C_2$ are positives constants depending only upon $M$.
 In  case that both the manifold and the metric are real analytic, the problem was solved by H. Donnelly and C. Fefferman \cite{DF1}, \cite{DF2}.
 For smooth metric the only known  upper bound result ($n\geq3$) is due to R. Hardt and L. Simon \cite{HS}.
 They proved that $$\Hc^{n-1}(\Nc_{u})\leq( c\sqrt{\lambda})^{c\sqrt{\lambda}}. $$ However this result doesn't seems to be optimal. 
Recently, different authors (\cite{Man} \cite{CM}, \cite{SZ}) obtained some lower bound  with polynomial decrease in $\lambda$.

The critical set of eigenfunctions on the other hand is not so well understood (one could look at \cite{Z} for a quick survey). 
Generically eigenfunctions are Morse functions (\cite{U}) and therefore the critical set consits in isolated points. 
Moreover, D. Jakobson and N. Nadirashvili \cite{JN} have shown that there exists in dimension two a sequence of eigenfunctions for which the number of critical points
 is uniformly bounded. However there exists simple examples for which the critical set has Hausdorff dimension $n-1$ : 
\begin{ex}\label{opt}
 Let $(N,g)$ be a $(n-1)-$dimensional manifold and define $M=\T^1\times N$  where $\T^1$ is the $1-$dimensionnal Torus with standard metric, and $M$ is equipped with the product metric.
The function $f_k(x,y)=\sin(2\pi kx)$ is an eigenfunction of $\Delta_M$ with eigenvalue $\lambda:=k^2$. The critical set, $\Cc_{f_k}$, of $f_k$ is therefore a set of dimension $n-1$. 
One should also note that $\Hc^{n-1}(\Cc_{f_k})\geq C\sqrt{\lambda}$, where $C$ depends only on $M$.
\end{ex}
It is also easy to find some surface of revolution with critical set of dimension $(n-1)$, see \cite{Z} p 35.
In the case of a critical set of dimension $n-1$ it seems interresting to obtained some upper bound on the $(n-1)-$dimensionnal Hausdorff measure. This is the goal of this paper.
We will show that :
 \begin{thm}\label{principal} Let $M$ be a $n$-dimensionnal, real analytic, compact, connected manifold with analytic metric. 
There exist $C>0$ depending only on $M$ such that for any non-constant solution $u$ to \eqref{fp} one has 
$$\Hc^{n-1}(\Cc_u)\leq C\sqrt{\lambda},$$ where $\Cc_u$ is the critical set of $u$.
\end{thm} 
\vspace{0,5cm}
% The upper bound  $C\sqrt{\lambda}$ in this theorem is sharp. This can be seen by considering the following example on the two dimensional flat torus.
%  Let $f_k(x,y)=\cos(2\pi kx)$, one has $\Delta f_k=4\pi^2k^2f_k=\lambda_kf_k$ and $\mathcal{H}^{1}\left(\Cc_{f_k}\right)=2k\geq C \sqrt{\lambda_k}.$
%  Unlike the nodal set, one cannot expect a lower bound of the $(n-1)$-dimensional Hausdorff measure on the critical set.
%  Indeed, consider, again on the 2-dimensional flat torus, the following eigenfunction : $f(x,y)=\sin(2\pi k_1x)\sin(2\pi k_2y)$ with $k_1$ and $k_2$ any integer. The set $\Cc_f$ is finite and then $\Hc^{1}(C_f)=0$,  
%( furthermore, D. Jakobson and N. Nadirashvili \cite{JN} have shown that there exists in dimension two a sequence of eigenfunctions for which the number of critical points
 %is uniformly bounded).  
 
  \vspace{0,3cm}%by 16

The main ingredient in the proof of our theorem is the following doubling inequality on gradient of eigenfunctions 
\begin{equation}\label{doubintro}
 \|\nabla u\|_{B_{2r}}\leq e^{C\sqrt{\lambda}}\|\nabla u\|_{B_r}.
\end{equation}

 This estimate is a consequence of a general Carleman-type inequality which we also use to study the vanishing order 
of solutions to the Schr\"odinger equation in a related paper \cite{B1}.

The paper is organised as follows. In the section 2 we deduce from \cite{B1} a Carleman estimate for the operator $\Delta +\lambda$ 
acting diagonally on vector valued functions.
 Using the compactness of $M$, this will allows us  to derive in section 3 doubling estimates \eqref{doubintro}  using 
standard method of quantitative uniqueness. 
In section 3 we use the method developped by H. Donnelly and  C. Fefferman to show our estimate on 
the measure of the critical set in the case that $M$ is an analytic manifold.
One should note that the framework of this paper follows closely \cite{B1} until section 3, with some obvious adaptations to the vectorial case. 
\section{Carleman estimates} 
First we give a Carleman estimate on the scalar operator $\Delta +W$  with $W$ of class $\Cc^1$, this can also be find in \cite{B1} 
and is write down here only for completness (and because of the electronic nature of this document).\newline
\non Fix $x_0$ in $M$,  and let : $r=r(x)=d(x,x_0)$ the Riemannian distance from $x_0$. We denote by $B_r(x_0)$ the geodesic ball centered at  $x_0$ of radius $r$.
We will denote  by $\|\cdot\|$ the  $L^2$ norm. %The goal of this section is to proove the following Carleman inequality on the operator $\Delta +W$ :
Recall that Carleman estimates are weighted integral inequalities with a weight function $e^{\tau\phi}$, where the function 
$\phi$ satisfy some convexity properties. 
Let us now define the weight function we will use.\newline  
For a fixed number $\varepsilon$ such that  $0<\varepsilon<1$ and $T_0<0$, we define the function $f$ on  $]-\infty,T_0[ $ by $f(t)=t-e^{\varepsilon t}$. 
One can check easily that,  for $|T_0|$ great enough, the function $f$ verifies the following properties:
\begin{equation}\label{f}
\begin{array}{lcr}
& 1-\varepsilon e^{\varepsilon T_0}\leq f^\prime(t)\leq 1&\forall t\in]-\infty,T_0[,\\
&\displaystyle{\lim _{t\rightarrow-\infty}-e^{-t}f^{\prime\prime}(t)}=+\infty. &
\end{array}
\end{equation}
Finally we define $\phi(x)=-f(\ln r(x))$. Now we can state the main result of this section:

\begin{thm}\label{tics}
 There exist positive constants $R_0, C,C_1,C_2$, which depend only on  $M$, such that, 
for any $\:W\in \Cc^1(M)$, $x_0\in M$, $u\in C^\infty_0(B_{R_0}(x_0)\setminus\{0\})$ and 
 $\tau \geq C_1\sqrt{\|W\|_{\mathcal{C}^1}}+C_2$, one has    
\begin{equation}\label{S}%\begin{array}{ccc}
C\left\|r^2e^{\tau\phi}\left(\Delta u +Wu \right)\right\|\geq \tau^\frac{3}{2}\left\|r^{\frac{\varepsilon}{2}}e^{\tau\phi}u\right\|
+ \tau^{\frac{1}{2}}\left\|r^{1+\frac{\varepsilon}{2}}e^{\tau\phi}\nabla u\right\|.
%\end{array}
\end{equation}

\non Moreover, if  $$\mathrm{supp}(u)\subset\{x\in M; r(x)\geq\delta>0\},$$ then
\begin{equation}\label{Siv2}\begin{array}{rcc}
C\left\|r^2e^{\tau\phi}\left(\Delta u +Wu \right)\right\|&\geq& \tau^\frac{3}{2}\left\|r^{\frac{\varepsilon}{2}}e^{\tau\phi}u\right\| \\
+\ \tau\delta\left\|r^{-1}e^{\tau\phi}u\right\|&+& \tau^{\frac{1}{2}}\left\|r^{1+\frac{\varepsilon}{2}}e^{\tau\phi}\nabla u\right\|.
\end{array}\end{equation}
\end{thm}
% \begin{Rq}
%  This inequality can be seen as a generalization of previous Carleman type estimates in the case that $W$ is a constant (see \cite{DF1}).
%  Indeed when $W=\lambda$ one has $\sqrt{\|W\|}_{\Cc^1}=\sqrt{\lambda}$.
%  The point is that since $W$ is $\Cc^1$ we will be allowed to integrate by parts, but then we have to take care of the derivatives of $W$.
%   \vspace{0,4mm}
% \end{Rq}
% \begin{Rq}
%  In the inequalities \eqref{S} and \eqref{Siv2} the gradient terms are not necessary to the purpose of this paper. 
% We choose to include them for a more general statement.
% \end{Rq}

\begin{proof}Hereafter $C$, $C_1$, $C_2$  and $c$ denote positive constants depending only upon $M$, though their values may change from one line to another. 
Without loss of generality, we  may suppose that all functions are real.   
We now introduce the polar geodesic coordinates $(r,\theta)$ near $x_0$. Using Einstein notation, the Laplace operator takes the form : 
$$r^2\Delta u=r^2\partial_r^2u+r^2\left(\partial_r\ln(\sqrt{\gamma})+\frac{n-1}{r}\right)\partial_ru+
\frac{1}{\sqrt{\gamma}}\partial_i(\sqrt{\gamma}\gamma^{ij}\partial_ju),$$
where $\displaystyle{\partial_i=\frac{\partial}{\partial\theta_i}}$ and for each fixed $r$,  $\ \gamma_{ij}(r,\theta)$  
is a metric on \: $\Sb^{n-1}$ and $\displaystyle{\gamma=\mathrm{det}(\gamma_{ij})}$.\newline
Since $(M,g)$ is smooth, we have for $r$ small enough :  
\begin{eqnarray}\label{m1}
\partial_r(\gamma^{ij})&\leq& C (\gamma^{ij})\ \ \ \mbox{(in the sense of tensors)}; 
\nonumber \\
|\partial_r(\gamma)|&\leq& C;\\
C^{-1}\leq\gamma&\leq& C.\nonumber 
 \end{eqnarray}

\non Set  $r=e^t$, we have $\displaystyle{\frac{\partial}{\partial r}=e^{-t}\frac{\partial}{\partial t}}$. 
Then the function $u$ is supported in \\ $]-\infty,T_0[\times\mathbb{S}^{n-1},$ 
where $|T_0|$ will be chosen large enough. In this new variables, we can write : 
$$e^{2t}\Delta u=\partial_t^2u+(n-2+\partial_t\mathrm{ln}\sqrt{\gamma})\partial_tu
+\frac{1}{\sqrt{\gamma}}\partial_i(\sqrt{\gamma}\gamma^{ij}\partial_ju).$$
The conditions (\ref{m1}) become
\begin{eqnarray}\label{m2}
\partial_t(\gamma^{ij})&\leq& Ce^t (\gamma^{ij})\nonumber\ \ \ \mbox{(in the sense of tensors)};  \\
|\partial_t(\gamma)|&\leq& Ce^t;\\
C^{-1}\leq\gamma &\leq & C. \nonumber
\end{eqnarray}
Now we introduce the conjugate operator : 
\begin{equation}\begin{array}{rcl}
L_\tau(u)&=&e^{2t}e^{\tau\phi}\Delta(e^{-\tau\phi}u)+e^{2t}Wu\\
               &=&\partial^2_tu+\left(2\tau f^\prime+n-2+\partial_t\mathrm{ln}\sqrt{\gamma}\right)\partial_tu\\
               &+&\left(\tau^2f^{\prime^2}+\tau f^{\prime\prime}+(n-2)\tau f^{\prime}+\tau\partial_t\mathrm{ln}\sqrt{\gamma}f^{\prime}\right)u\\
               &+&\Delta_\theta u+e^{2t}Wu,
     \end{array}\end{equation}
with $$\Delta_\theta u=\frac{1}{\sqrt{\gamma}}\partial_i\left(\sqrt{\gamma}\gamma^{ij}\partial_ju\right).$$

\noindent It will be useful for us to introduce the following $L^2$ norm on $]-\infty,T_0[\times\Sb^{n-1} $: 
$$\|V\|_f^2=\int_{]-\infty,T_0[\times\Sb^{n-1}} V^2\sqrt{\gamma}{f^{\prime}}^{-3}dtd\theta,$$ where $d\theta$ is the usual measure on $\Sb^{n-1}$.
The corresponding inner product is denoted by  $\left\langle\cdot,\cdot\right\rangle_f$ , \emph{i.e} $$\langle u,v\rangle_f = \int uv\sqrt{\gamma}{f^{\prime}}^{-
3}dtd\theta.$$ 
\noindent We will estimate from below $\|L_\tau u\|^2_f$ by using elementary algebra and integrations by parts. We are concerned, in the computation, by the power of $\tau$ 
and  exponenial decay when $t$ goes to $-\infty$. First by triangular inequality one has
\begin{equation}
\|L_\tau(u)\|_f\geq I-I\!I ,
\end{equation}
with 
\begin{equation}\begin{array}{rcl}
I&=& \left\|\partial^2_tu+2\tau f^\prime\partial_tu+\tau^2f^{\prime^2}u+e^{2t}Wu+\Delta_\theta u\right\|_f,\\
I\!I&=&\left\|\tau f^{\prime\prime}u+(n-2)\tau f^\prime u+\tau\partial_t\mathrm{ln}\sqrt{\gamma}f^{\prime}u\right\|_f\\
&+&\left\|(n-2)\partial_tu+\partial_t\ln\sqrt{\gamma}\partial_tu\right\|_f.
\end{array}\end{equation}
We will be able to absorb $I\!I$ later. Then we compute $I^2$ :

$$I^2=I_1+I_2+I_3,$$ with 

\begin{equation}\begin{array}{rcl}
I_1&=&\|\partial^2_tu+(\tau^2f^{\prime^2}+e^{2t}W)u+\Delta_\theta u\|_f^2\\
I_2&=&\|2\tau f^\prime\partial_tu\|_f^2\\%&&\nonumber\\
I_3&=&2\left<2\tau f^{\prime}\partial_tu ,\partial^2_tu+\tau^2f^{\prime^2}u+e^{2t}Wu+\Delta_\theta u\right>_f\label{I}
\end{array}\end{equation}
           
\non In order to compute $I_3$ we write it in a convenient way: 
\begin{equation}
 I_3=J_1+J_2+J_3,
\end{equation}
where the integrals $J_i$ are defined by :
\begin{equation}\label{I3}
\begin{array}{rcl}
J_1&=&2\tau \int f^\prime \partial_t(|\partial_tu|^2)f^{\prime^{-3}}\sqrt{\gamma}dtd\theta\\
J_2&=&4\tau\int f^\prime\partial_tu\partial_i\left(\sqrt{\gamma}\gamma^{ij}\partial_ju\right)f^{\prime^{-3}}dtd\theta\\
J_3&=&\int\left(2\tau^3 (f^\prime)^3+2\tau f^\prime e^{2t}W\right)2u\partial_tuf^{\prime^{-3}}\sqrt{\gamma}dtd\theta.
\end{array}
\end{equation}
\noindent Now we will use integration by parts to estimate each terms of \eqref{I3}. 
Note that $f$ is radial and that $2\partial_tu\partial_t^2u=\partial_t(|\partial_tu|^2)$. We find that  :
\begin{equation*}\begin{array}{rcl}
J_1&=& \int\left(4\tau f^{\prime\prime}\right)|\partial_tu|^2f^{\prime^{-3}}\sqrt{\gamma}dtd\theta\\
&-&\int2\tau f^\prime\partial_t\mathrm{ln}\sqrt{\gamma}|\partial_{t}u|^2f^{\prime^{-3}}\sqrt{\gamma}dtd\theta.\end{array}
\end{equation*}
The conditions \eqref{m2} imply that $|\partial_t\ln\sqrt{\gamma}|\leq Ce^t$. Then properties \eqref{f} on $f$ gives, for large $|T_0|$ 
that $|\partial_t\ln\sqrt{\gamma}|$ is small compared to $|f^{\prime\prime}|$. Then one has   
\begin{equation}\label{J_1}J_1\geq -c\tau\int |f^{\prime\prime}|\cdot|\partial_tu|^2f^{\prime^{-3}}\sqrt{\gamma}dtd\theta.\end{equation}

\noindent Now in order to estimate $J_2$ we first integrate by parts with respect to $\partial_i$ : 
\begin{equation*}\begin{array}{rcl}
J_2
   &=&-2\int2\tau f^{\prime}\partial_t\partial_iu\gamma^{ij}\partial_juf^{\prime^{-3}}\sqrt{\gamma}dtd\theta.                  
     \end{array}
\end{equation*}            
Then we integrate by parts with respect to $\partial_t$. We get : 
\begin{equation*}\begin{array}{rcl}
J_2&=&-4\tau\int f^{\prime\prime}\gamma^{ij}\partial_iu\partial_juf^{\prime^{-3}}\sqrt{\gamma}dtd\theta\\
&+&\int2\tau f^{\prime}\partial_t\mathrm{ln}\sqrt{\gamma}\gamma^{ij}\partial_iu\partial_juf^{\prime^{-3}}\sqrt{\gamma}dtd\theta\\
&+&\int2\tau f^\prime\partial_t(\gamma^{ij})\partial_iu\partial_juf^{\prime^{-3}}\sqrt{\gamma}dtd\theta.
\end{array}
\end{equation*}
We denote  $|D_\theta u|^2=\partial_iu\gamma^{ij}\partial_ju$. Now using that $-f^{\prime\prime}$ is non-negative and $\tau$ is large, 
the conditions \eqref{f}  and \eqref{m2} gives for $|T_0|$ large enough:
\begin{equation}\label{J_2}
 J_2\geq 3\tau\int|f^{\prime\prime}|\cdot|D_\theta u|^2f^{\prime^{-3}}\sqrt{\gamma}dtd\theta.
\end{equation}
Similarly computation of $J_3$ gives :
\begin{equation}\label{J_31}\begin{array}{rcl}
J_3&=&-2\int\tau^3\partial_t\mathrm{ln}(\sqrt{\gamma})u^2\sqrt{\gamma}dtd\theta\\
&-&\int(4f^\prime-4 f^{\prime\prime}+2f^{\prime}\partial_t\ln\sqrt{\gamma} )\tau e^{2t}Wu^2f^{\prime^{-3}}\sqrt{\gamma}dtd\theta\\
&-&\int2\tau f^\prime e^{2t}\partial_tW|u|^2f^{\prime^{-3}}\sqrt{\gamma}dtd\theta.
\end{array}\end{equation}
\noindent Now we assume that \begin{equation}\tau\geq C_1\sqrt{\|W\|_{\mathcal{C}^1}}+C_2.\end{equation}
 From  \eqref{f} and \eqref{m2} one can see that if $C_1$, $C_2$ and $|T_0|$ are large enough, then\:\newline
\begin{equation}\label{J_3}
J_3\geq-c\tau^3\int e^t|u|^2f^{\prime^{-3}}\sqrt{\gamma}dtd\theta.
\end{equation}

 \noindent Thus far, using \eqref{J_1},\eqref{J_2} and \eqref{J_3}, we have : 
 \begin{equation}\label{I_3}\begin{array}{rcl}
 I_3 &\geq & 3\tau \int\left|f^{\prime\prime}\right|\left|D_\theta u\right|^2f^{\prime^{-3}}\sqrt{\gamma}dtd\theta -c\tau^3\int e^t|u|^2f^{\prime^{-3}}\sqrt{\gamma}dtd\theta\\
 &-&c\tau\int\left|f^{\prime\prime}\right|\left|\partial_t u\right|^2f^{\prime^{-3}}\sqrt{\gamma}dtd\theta.
 \end{array}\end{equation}

\noindent Now we consider  $I_1$ :
$$I_1=\left\|\partial^2_tu+\left(\tau^2
f^{\prime^2}+e^{2t}W\right)u+\Delta_\theta u\right\|_f^2 .$$
Let $\rho>0$ a small number to be chosen later. Since  $|f^{\prime\prime}|\leq1$ and $\tau\geq1$, we have : \newline
\begin{equation}I_1\geq\frac{\rho}{\tau}I_1^\prime,\end{equation}
where $I_1^\prime$ is defined by :
\begin{equation}I_1^\prime=\left\|\sqrt{|f^{\prime\prime}|}\left[\partial^2_tu+\left(\tau^2f^{\prime^2}+e^{2t}W\right)u+\Delta_\theta u\right]\right\|_f^2 \end{equation}
and one has \begin{equation}I_1^\prime=K_1+K_2+K_3,\end{equation} 
with 
\begin{equation}\label{Ki}
\begin{array}{rcl}
  K_1&=&\left\|\sqrt{|f^{\prime\prime}|}\left(\partial_t^2u+\Delta_\theta u\right)\right\|_f^2, \\
K_2&=&\left\|\sqrt{|f^{\prime\prime}|}\left(\tau^2f^{\prime^2}+e^{2t}W\right)u\right\|_f^2,  \\
K_3&=&2\left\langle\left(\partial_t^2u+\Delta_\theta u\right)\left|f^{\prime\prime}\right|,\left(\tau^2f^{\prime^2}+e^{2t}W\right)u\right\rangle_f.
\end{array}\end{equation}
Integrating by parts gives : 
\begin{equation}\label{K_3}\begin{array}{rcl}
K_3&=&2\int f^{\prime\prime}\left(\tau^2f^{\prime^2}+e^{2t}W\right)|\partial_tu|^2f^{\prime^{-3}}\sqrt{\gamma}dtd\theta\\
&+&2\int \partial_t\left[f^{\prime\prime}\left(\tau^2f^{\prime^2}+e^{2t}W \right)\right]\partial_tuu\sqrt{\gamma}f^{\prime^{-3}}dtd\theta \\
&-&6\int \left(f^{\prime\prime^2}f^{\prime^{-1}}\left(\tau^2f^{\prime^2}+e^{2t}W \right)\right)\partial_tuu\sqrt{\gamma}f^{\prime^{-3}}dtd\theta \\
&+&2\int f^{\prime\prime}\left(\tau^2f^{\prime^2}+e^{2t}W\right)\partial_t\mathrm{ln}\sqrt{\gamma}\partial_tuuf^{\prime^{-3}}\sqrt{\gamma}dtd\theta\\
&+&2\int f^{\prime\prime}\left(\tau^2f^{\prime^2}+e^{2t}W\right)|D_\theta u|^2f^{\prime^{-3}}\sqrt{\gamma}dtd\theta\\
&+&2\int f^{\prime\prime}e^{2t}\partial_iW\cdot\gamma^{ij}\partial_juuf^{\prime^{-3}}\sqrt{\gamma}dtd\theta.
\end{array}\end{equation}
 The condition $\tau\geq C_1\sqrt{\|W\|_{\Cc^1}}+C_2$ implies, $$|\partial_iW\gamma^{ij}\partial_juu|\leq c\tau^2(|D_\theta u|^2+|u|^2).$$
Now since  $2\partial_tuu\leq u^2+|\partial_tu|^2$, we can  use conditions \eqref{f} and \eqref{m2} to get

\begin{equation}
 K_3\geq-c\tau^2\int|f^{\prime\prime}|\left(|\partial_tu|^2+|D_\theta u|^2+|u|^2\right)f^{\prime^{-3}}\sqrt{\gamma}dtd\theta\\
\end{equation}
We also have 

\begin{equation}
K_2\geq c\tau^4\int|f^{\prime\prime}||u|^2f^{\prime^{-3}}\sqrt{\gamma}dtd\theta
\end{equation}
and since  $K_1\geq 0$ ,  
\begin{equation}\label{I_1}\begin{array}{rcl}
I_1&\geq&-\rho c\tau\int|f^{\prime\prime}|\left(|\partial_t u|^2+|D_\theta u|^2\right)f^{\prime^{-3}}\sqrt{\gamma}dtd\theta\\
&+&C\tau^{3}\rho\int|f^{\prime\prime}||u|^2f^{\prime^{-3}}\sqrt{\gamma}dtd\theta.
\end{array}\end{equation}

\noindent Then using  \eqref{I_3} and \eqref{I_1} 
\begin{equation}\label{pro}\begin{array}{rcl}
I^2  &\geq & 4\tau^2\| f^\prime\partial_tu\|_f^2+3\tau\int|f^{\prime\prime}||D_\theta u|^2f^{\prime^{-3}}\sqrt{\gamma}dtd\theta \\
&+&C\tau^{3}\rho\int|f^{\prime\prime}||u|^2f^{\prime^{-3}}\sqrt{\gamma}dtd\theta-c\tau^3\int e^t|u|^2f^{\prime^{-3}}\sqrt{\gamma}dtd\theta\\
%&-&c\tau^3\int|f^{\prime\prime}|e^{2t}u^2f^{\prime^{-3}}\sqrt{\gamma}dtd\theta-c\tau^3\int|\partial_t\mathrm{ln}\sqrt{\gamma}|u^2f^{\prime^{-3}}\sqrt{\gamma}dtd\theta\\
&-&\rho c\tau\int|f^{\prime\prime}|\left(|u|^2+|\partial_tu|^2+|D_\theta u|^2\right)f^{\prime^{-3}}\sqrt{\gamma}dtd\theta.\\
&-&c\tau\int|f^{\prime\prime}||\partial_tu|^2f^{\prime^{-3}}\sqrt{\gamma}dtd\theta
\end{array}\end{equation}

\noindent
Now one needs to check that every non-positive term in the right hand side of \eqref{pro} can be absorbed in the first three terms. 
\\ First fix $\rho$ small enough such that
$$\rho c\tau\int|f^{\prime\prime}|\cdot|D_\theta u|^2{f^{\prime}}^{-3}\sqrt{\gamma}dtd\theta\leq 2\tau\int|f^{\prime\prime}|\cdot|D_\theta u|^2f^{\prime^{-3}}\sqrt{\gamma}dtd\theta$$
where $c$ is the constant appearing in \eqref{pro}. The other  terms in the last integral of \eqref{pro} can then be absorbed by comparing powers of $\tau$ 
(for $C_2$ large enough). 
Finally since conditions \eqref{f} imply that $e^t$ is small compared to $|f^{\prime\prime}|$, 
we can absorb   $-c\tau^3e^t|u|^2$ in $C\tau^{3}\rho|f^{\prime\prime}||u|^2$.

\noindent Thus we obtain :
\begin{equation}\label{ssu}\begin{array}{rcl}
 I^2 &\geq &C\tau^2\int|\partial_t u|^2f^{\prime^{-3}}\sqrt{\gamma}dtd\theta+C\tau\int|f^{\prime\prime}||D_\theta u|^2f^{\prime^{-3}}\sqrt{\gamma}dtd\theta\\
&+ &C\tau^{3}\int|f^{\prime\prime}||u|^2f^{\prime^{-3}}\sqrt{\gamma}dtd\theta
\end{array}\end{equation}

\noindent As before, we can check that $I\!I$ can be absorbed in $I$ for  $|T_0|$ and $\tau$ large enough. 
Then we obtain \begin{equation}\label{bla}\|L_\tau u\|_f^2\geq C\tau^3\|\sqrt{|f^{\prime\prime}|} u\|_f^2+C\tau^2\|\partial_t u\|_f^2
+C\tau\|\sqrt{|f^{\prime\prime}|}D_\theta u\|_f^2 .\end{equation}
Note that, since $\tau$ is large and $\sqrt{|f^{\prime\prime}|}\leq1$, one has
 \begin{equation}\|L_\tau u\|_f^2\geq C\tau^3\|\sqrt{|f^{\prime\prime}|} u\|_f^2+c\tau\|\sqrt{|f^{\prime\prime}|}\partial_t u\|_f^2
+C\tau\|\sqrt{|f^{\prime\prime}|}D_\theta u\|_f^2,\end{equation}
and the constant $c$ can be choosen arbitrary smaller than $C$. 
If we set  $v=e^{-\tau\phi}u$, then we have 
\begin{equation*}
 \begin{array}{rcl}\|e^{2t}e^{\tau\phi}(\Delta v+Wv)\|_f^2&\geq& C\tau^3\|\sqrt{|f^{\prime\prime}|}e^{\tau\phi} v\|_f^2
-c\tau^3\|\sqrt{|f^{\prime\prime}|}f^\prime e^{\tau\phi} v\|_f^2\\
+
\frac{c}{2}\tau\|\sqrt{|f^{\prime\prime}|}e^{\tau\phi}\partial_t v\|_f^2&+&C\tau\|\sqrt{|f^{\prime\prime}|}e^{\tau\phi} D_\theta v\|_f^2\end{array}.
\end{equation*}
Finally since $f^\prime$ is close to 1 one can absorb the negative term to obtain
 
\begin{equation}\begin{array}{rcl}\|e^{2t}e^{\tau\phi}(\Delta v+Wv)\|_f^2&\geq& C\tau^3\|\sqrt{|f^{\prime\prime}|}e^{\tau\phi} v\|_f^2\\
+
C\tau\|\sqrt{|f^{\prime\prime}|}e^{\tau\phi}\partial_t v\|_f^2&+&C\tau\|\sqrt{|f^{\prime\prime}|}e^{\tau\phi} D_\theta v\|_f^2\end{array}.\end{equation}
It remains to get back to the usual $L^2$ norm. First note that since $f^\prime$ is close to 1 \eqref{f},
 we can get the same estimate without the term $(f^\prime)^{-3}$ in the integrals. Recall that in polar coordinates $(r,\theta)$ the volume element 
is $r^{n-1}\sqrt{\gamma}drd\theta$, we can deduce from \eqref{ssu} by substitution that : 
\begin{equation}
 \begin{array}{rcl}
  \|r^2e^{\tau\phi}(\Delta v+Wv)r^{-\frac{n}{2}}\|^2&\geq&C\tau^3\|r^\frac{\varepsilon}{2}e^{\tau\phi}vr^{-\frac{n}{2}}\|^2\\
&+&C\tau\|r^{1+\frac{\varepsilon}{2}}e^{\tau\phi}\nabla vr^{-\frac{n}{2}}\|^2.
 \end{array}
\end{equation}
Finally one can get rid of the term $r^{-\frac{n}{2}}$ by replacing $\tau$ with $\tau+\frac{n}{2}$. Indeed from 
$e^{\tau\phi}r^{-\frac{n}{2}}=e^{(\tau+\frac{n}{2})\phi}e^{-\frac{n}{2}r^\varepsilon}$ one can check easily that, for $r$ small enough  
$$\frac{1}{2}e^{(\tau+\frac{n}{2})\phi}\leq e^{\tau\phi}r^{-\frac{n}{2}}\leq e^{(\tau+\frac{n}{2})\phi}.$$ 

\non This achieves the  proof of the first part of theorem \ref{tics}.\\

\non Now suppose that  $\mathrm{supp}(u)\subset\{x\in M; r(x)\geq\delta>0\}$ and define $T_1=\ln\delta$.\newline

\noindent Cauchy-Schwarz inequality apply to $$\int\partial_t(u^2)e^{-t}\sqrt{\gamma}dtd\theta=2\int u\partial_tue^{-t}\sqrt{\gamma}dtd\theta$$ 
gives
\begin{equation}\label{d1}
 \int\partial_t(u^2)e^{-t}\sqrt{\gamma}dtd\theta\leq 2\left(\int\left(\partial_tu\right)^2e^{-t}\sqrt{\gamma}dtd\theta \right)^{\frac{1}{2}}
\left(\int u^2e^{-t}\sqrt{\gamma}dtd\theta\right)^{\frac{1}{2}}.
\end{equation}
On the other hand, integrating by parts gives
 \begin{equation}
      \int\partial_t(u^2)e^{-t}\sqrt{\gamma}dtd\theta = \int u^2e^{-t}\sqrt{\gamma}dtd\theta-\int u^2e^{-t}\partial_t(\ln(\sqrt{\gamma}))\sqrt{\gamma}dtd\theta.
     \end{equation}
Now since  $|\partial_t\ln\sqrt{\gamma}|\leq Ce^t$ for $|T_0|$ large enough we can deduce :
\begin{equation}\label{d2}
 \int\partial_t(u^2)e^{-t}\sqrt{\gamma}dtd\theta \geq c  \int u^2e^{-t}\sqrt{\gamma}dtd\theta.
\end{equation}
Combining \eqref{d1} and  \eqref{d2} gives  
\begin{eqnarray*}\label{prout}
 c^2 \int u^2e^{-t}\sqrt{\gamma}dtd\theta&\leq& 4\int\left(\partial_tu\right)^2e^{-t}\sqrt{\gamma}dtd\theta\\
&\leq&4e^{-T_1}\int\left(\partial_t u\right)^2\sqrt{\gamma}dtd\theta.\end{eqnarray*}

\noindent Finally, droping all terms except  $\tau^2\int|\partial_t u|^2f^{\prime^{-3}}\sqrt{\gamma}dtd\theta$  in \eqref{ssu}  gives :

\begin{eqnarray*}
C^{\prime}I^2\geq \tau^2 \delta^2 \|e^{-t}u\|_f^2.
\end{eqnarray*}\vspace{0,5cm}
Inequality \eqref{ssu} can then be replaced by :  
\begin{equation}\begin{array}{rcl}
 I^2 &\geq &C\tau^2\int|\partial_t u|^2f^{\prime^{-3}}\sqrt{\gamma}dtd\theta+C\tau\int|f^{\prime\prime}|\cdot|D_\theta u|^2f^{\prime^{-3}}\sqrt{\gamma}dtd\theta\\
&+ &C\tau^{3}\int|f^{\prime\prime}|\cdot|u|^2f^{\prime^{-3}}\sqrt{\gamma}dtd\theta+C\tau^2 \delta^2\int|u|^2{f^\prime}^{-3}\sqrt{\gamma}dtd\theta.
\end{array}\end{equation}
The rest of the proof follows in a similar way than the first part. 
\end{proof}

% \begin{Rq}  From now on, we will assume that $f(t)=t-e^{\varepsilon t}$, $\varepsilon>0$,  in the above theorem.  Then one can see that every estimate is still valid if we replace the norm $\|\cdot\|$ by the usual $L^2$ norm. 
% This can be seen by adding to $\tau$ a constant depending only on the dimension $n$.
%  We still denote by $\|\cdot\|$ the standard $L^2$ norm.\end{Rq}

Now we will establish a Carleman estimate for the operator  $\Delta+\lambda $ acting on vector functions, which will be useful in the next section. 
For $U\in\mathcal{C}^{\infty}_0(B_{R_0}(x_0)\setminus\{x_0\},\mathbb{R}^m)$, applying (\ref{S})  to each components $U^i$ of $U$ and summing gives : 
\begin{cor}\label{carv}%Under the same assumptions as in theorem \ref{tics} one has, 
There exist non-negative constants $R_0, C,C_1$, which depend only on  $M$ and $\varepsilon$, such that : \newline  
$\forall x_0\in M,\:\forall\: U\:\in\mathcal{C}^{\infty}_0(B_{R_0}(x_0)\setminus\{x_0\},\mathbb{R}^m),\ \forall\  \tau\geq C_1\sqrt{\lambda}, $
% $\forall W\in \Cc^1(M),,\forall u\in C^\infty_0(B_{R_0}(x_0)\setminus\{0\}), \forall \tau \geq C_1\sqrt{\|W\|_{\mathcal{C}^1}}+C_2, $   
 
\begin{eqnarray}C\left\|r^2e^{-\tau\phi}\left(\Delta U +\lambda U \right)\right\|&\geq& \tau^\frac{3}{2}\left\|r^{\frac{\varepsilon}{2}} e^{-\tau\phi}U\right\| \nonumber\\
& +& \tau^{\frac{1}{2}}\left\|r^{1+\frac{\varepsilon}{2}}e^{-\tau\phi}\nabla U\right\|
\end{eqnarray}
Moreover, 
$$\mathrm{If}\:\:\:\mathrm{supp}(U)\subset\{x\in M; r(x)\geq\delta>0\},$$ then
\begin{eqnarray}\label{Sivv2}
C\left\|r^2e^{-\tau\phi}\left(\Delta U+\lambda U\right)\right\|&\geq& \tau^\frac{3}{2}\left\|r^{\frac{\varepsilon}{2}}e^{-\tau\phi}U\right\| \nonumber\\
+\ \tau\delta\left\|r^{-1}e^{-\tau\phi}U\right\|&+& \tau^{\frac{1}{2}}\left\|r^{1+\frac{\varepsilon}{2}}e^{-\tau\phi}\nabla U\right\|.
\end{eqnarray}
\end{cor}

\section{Doubling inequality}
In this section we intend to prove a doubling property for gradient of eigenfunctions.
First we establish a  three sphere theorem  :
\begin{prop}[Three spheres theorem]
There exist non-negative constants $R_0$, $c$ and $0<\alpha<1$  wich depend only on $M$ such that, if $u$ is a solution to \eqref{fp} one has :\newline 
 $\forall R;\  0<R<2R<R_0, \forall x_0\in M, $
\begin{equation}\label{ttc}
\|\nabla u\|_{B_R(x_0)}\leq e^ {c \sqrt{\lambda}}\|\nabla u\|_{B_{\frac{R}{2}}(x_0)}^\alpha\|\nabla u\|_{B_{2R}(x_0)}^{1-\alpha}%\:,\:\:
\end{equation}
\end{prop}

\begin{proof}
Let $x_0$ a point in $M$ and $(x_1,x_2,\cdots,x_n)$ local coordinates around $x_0$. Let $u$ be a solution 
to \eqref{fp} and define $V=(\frac{\partial u}{\partial x_1},\cdots,\frac{\partial u}{x_n}).$ 
Let  $R_0>0$ as in theorem  (\ref{S}) and $R$ such that $0<R<2R<R_0$. We still denote $r(x)$ the riemannian distance beetween $x$ and $x_0$. We also denote by $B_r$ the geodesic ball centered at $x_0$ of radius $r$. If $v$ is a  function defined in a neigborhood of $x_0$, we denote by  $\|v\|_R$ the $L^2$ norm of $v$ on $B_R$ and by $\|v\|_{R_1,R_2}$ the $L^2$ norm of $v$ on the set  
 $A_{R_1,R_2}:=\{x\in M ;\:R_1\leq r(x)\leq R_2\}$.  
Let $\psi\in\Cc^{\infty}_0(B_{2R})$, $0\leq\psi\leq1$, a radial function with the following properties :
\begin{itemize}
%\item[$\bullet$] $0\leq \psi \leq 1$,
\item[$\bullet$] $\psi(x)=0$ if $r(x)<\frac{R}{4}$ or if $r(x)>\frac{5R}{3}$,
\item[$\bullet$] $\psi(x)=1$ if $\frac{R}{3}<r(x)<\frac{3R}{2}$,
\item[$\bullet$] $|\nabla\psi(x)|\leq \frac{C}{R}$, $|\nabla^{2}\psi(x) |\leq \frac{C}{R^2}$.
\end{itemize}

We recall that $\phi(x)=-\ln r(x)+r(x)^\varepsilon$.\vspace{0,4cm}\newline 
First  apply $\partial_k$ to each side of \eqref{fp} to get $$\Delta \partial_ku-[\Delta,\partial_k]u=\partial_ku$$ 
where $[\Delta,\partial_k]$ is a second order operator  with no zero order term and with coefficients depending only of $M$. 
The function $V=(\frac{\partial u}{\partial x_1},\cdots,\frac{\partial u}{x_n})$ is therefore a solution of the system  :

\begin{equation}\label{S2}\Delta V+\lambda V -AV-B\cdot\nabla V=0\end{equation} where $A$ and $B$ depend only on the metric $g$ of $M$ and its derivatives.
Now we apply the Carleman estimate  \eqref{Sivv2} to the function  $\psi V$ with $f(t)=t-e^{\varepsilon t}$. We get : 

\begin{eqnarray*}C\left\|r^2e^{\tau\phi}\left(\Delta(\psi V) +\lambda\psi V \right)\right\|&\geq& \tau^\frac{3}{2}\left\|r^{\frac{\varepsilon}{2}}e^{\tau\phi}\psi V\right\| \nonumber\\
+\ \tau R\left\|r^{-1}e^{\tau\phi}\psi V\right\|&+& \tau^{\frac{1}{2}}\left\|r^{1+\frac{\varepsilon}{2}}e^{\tau\phi}\nabla(\psi V)\right\|.
\end{eqnarray*}

\non Using that $V$ is a solution of \eqref{S2}, we have : 
 \begin{eqnarray*}
C\left\|r^2e^{\tau\phi}\left(\psi AV +\psi B\cdot\nabla V+2\nabla V\cdot\nabla\psi+\Delta\psi V\right)\right\|&\geq& \tau^\frac{3}{2}\left\|r^{\frac{\varepsilon}{2}}e^{\tau\phi}\psi V\right\| \nonumber\\
+\ \tau R\left\|r^{-1}e^{\tau\phi}\psi V\right\|+ \tau^{\frac{1}{2}}\left\|r^{1+\frac{\varepsilon}{2}}e^{\tau\phi}\nabla(\psi V)\right\|&& \nonumber
\end{eqnarray*}
Now  from triangular inequality we get
\begin{eqnarray*}
&C\left\|r^2e^{\tau\phi}\left(\Delta \psi V+2\nabla V\cdot\nabla\psi\right)\right\|\geq \tau^{\frac{3}{2}}\left\|r^{\frac{\varepsilon}{2}}e^{\tau\phi}\psi V\right\|-C\left\|r^2e^{\tau\phi}\psi AV\right\| &\\
&+\ \tau R\left\|r^{-1}e^{\tau\phi}\psi V\right\|+\tau^{\frac{1}{2}}\left\|r^{1+\frac{\varepsilon}{2}}e^{-\tau\phi}\nabla(\psi V)\right\|-C\left\|r^2e^{\tau\phi}\psi B\cdot\nabla V\right\| &
\end{eqnarray*}
and 
\begin{eqnarray*}\tau^{\frac{1}{2}}\left\|r^{1+\frac{\varepsilon}{2}}e^{\tau\phi}\nabla(\psi V)\right\|&\geq&\tau^{\frac{1}{2}}\left\|r^{1+\frac{\varepsilon}{2}}e^{\tau\phi}\psi\nabla V\right\|
-\tau^{\frac{1}{2}}\left\|r^{1+\frac{\varepsilon}{2}}e^{\tau\phi}\nabla\psi V\right\|  \\
&\geq&\tau^{\frac{1}{2}}\left\|r^{1+\frac{\varepsilon}{2}}e^{\tau\phi}\psi\nabla V\right\| -\tau^{\frac{1}{2}}\left\|r^{\frac{\varepsilon}{2}}e^{\tau\phi} V\right\| 
\end{eqnarray*}
Then for  $\tau$  great enough and for sufficient small $R_0$ ,
\begin{eqnarray}\label{cutof}C\left\|r^2e^{\tau\phi}\left(\Delta \psi V+2\nabla V\cdot\nabla\psi\right)\right\|&\geq& \tau^\frac{3}{2}\left\|r^{\frac{\varepsilon}{2}}e^{-\tau\phi}\psi V\right\| \nonumber\\
+\ \tau R\left\|r^{-1}e^{-\tau\phi}\psi V\right\|&+& \tau^{\frac{1}{2}}\left\|r^{1+\frac{\varepsilon}{2}}e^{-\tau\phi}\psi\nabla V\right\|.
\end{eqnarray}
In particular we have  :
\begin{equation*}C\left\|r^2e^{\tau\phi}\left(\Delta \psi V+2\nabla V\cdot\nabla\psi\right)\right\|\geq\tau\left\|e^{\tau\phi}\psi V\right\|\end{equation*}

\non Assume that $\tau \geq 1$, and use properties of $\psi$ to get  : 
\begin{eqnarray}\label{3c1}
\|e^{\tau\phi}V\|_{\frac{R}{3},\frac{3R}{2}} &\leq  &C\left(\|e^{\tau\phi}V\|_{\frac{R}{4},\frac{R}{3}}+\|e^{\tau\phi}V\|_{\frac{3R}{2},\frac{5R}{3}}\right) \nonumber\\
&+&C\left(R\|e^{\tau\phi}\nabla V\|_{\frac{R}{4},\frac{R}{3}}+R\|e^{\tau\phi}\nabla V\|_{\frac{3R}{2},\frac{5R}{3}}\right).\end{eqnarray}

\non Furthermore  as $\phi$ is radial and decreasing, 
$$\begin{array}{rcl}
\|e^{\tau\phi}V\|_{\frac{R}{3},\frac{3R}{2}}&\leq&
 C\left(e^{\tau\phi(\frac{R}{4})}\|V\|_{\frac{R}{4},\frac{R}{3}}+e^{\tau\phi(\frac{3R}{2})}\|V\|_{\frac{3R}{2},\frac{5R}{3}}\right)\\&+&C\left(Re^{\tau\phi(\frac{R}{4})}\|\nabla V\|_{\frac{R}{4},\frac{R}{3}}+Re^{\tau\phi(\frac{3R}{2})}\|\nabla V\|_{\frac{3R}{2},\frac{5R}{3}}\right).
\end{array}$$
Now we recall the following elliptic estimates : since $V$ satisfies \eqref{S2} then   hard to see that : \begin{equation}\label{lm1}\|\nabla V\|_{(1-a)r}\leq C\left(\frac{1}{(1-a)R}+\sqrt{\lambda}\right)\|V\|_{B_R}, \:\:\ \mathrm{for}\  \:0<a<1   \end{equation}
As $\|e^{\tau\phi}\nabla V\|_{\frac{R}{4},\frac{R}{3}}$ is bounded by $\|e^{\tau\phi}\nabla V\|_{\frac{R}{3}}$, using the formula \eqref{lm1} gives :

$$e^{\tau\phi(\frac{R}{4})}\|\nabla V\|_{\frac{R}{4},\frac{R}{3}}\leq C\left(\frac{1}{R}+\sqrt{\lambda}\right)e^{\tau\phi(\frac{R}{4})}\|V\|_{\frac{R}{2}},$$
Simiraly, we have also, 
$$e^{\tau\phi(\frac{3R}{2})}\|\nabla V\|_{\frac{3R}{2},\frac{5R}{3}}\leq C\left(\frac{1}{R}+\sqrt{\lambda}\right)e^{\tau\phi(\frac{3R}{2})}\|V\|_{2R}.$$
Using properties of  $\phi$ :
$$\|e^{\tau\phi}V\|_{\frac{R}{3},\frac{3R}{2}}\geq \|e^{\tau\phi}V\|_{\frac{R}{3},R}\geq e^{\tau\phi(R)}\|V\|_{\frac{R}{3},R}.$$
Using  (\ref{3c1}) one has :
 \begin{equation*}\label{ind}
\|V\|_{\frac{R}{3},R} \leq C\sqrt{\lambda}\left( e^{\tau(\phi(\frac{R}{4})-\phi(R))}\|V\|_{\frac{R}{2}}+e^{\tau(\phi(\frac{3R}{2})-\phi(R))}\|V\|_{2R}\right)
  \end{equation*}
%where $C_{\lambda}=C\sqrt{\lambda}$.\newline
  Let $A_R=\phi(\frac{R}{4})-\phi(R)$ and  $B_R=-(\phi(\frac{3R}{2})-\phi(R))$.
  Because of the properties of  $\phi$, we have $0<C_1\leq A_R\leq C_2$ and $0<C_1\leq B_R\leq C_2$ where $C_1$ and $C_2$ don't depend on $R$. 
We may assume that $C\sqrt{\lambda}\geq 2$. We can add $\|V\|_{\frac{R}{3}}$ to each member and bound it in the right hand side by  
$C\sqrt{\lambda}e^{\tau A}\|V\|_{\frac{R}{2}}$. Then replacing $C$ by $2C$ gives :
  
\begin{eqnarray}\label{3c2}
\|V\|_{R} &\leq &C\sqrt{\lambda} e^{\tau A}\|V\|_{\frac{R}{2}}+\|V\|_{\frac{R}{3}}+C_\lambda e^{-\tau B}\|V\|_{2R}\\
\|V\|_{R} &\leq & C\sqrt{\lambda}\left(e^{\tau A}\|V\|_{\frac{R}{2}}+e^{-\tau B}\|V\|_{2R}\right).
  \end{eqnarray} 
Now we want to find $\tau$ such that $$C\sqrt{\lambda} e^{-\tau B}\|V\|_{2R}\leq \frac{1}{2}\|V\|_{R}$$
wich is true for $\tau \geq -\frac{1}{B}\ln\left(\frac{1}{2C\sqrt{\lambda} }\frac{\|V\|_R}{\|V\|_{2R}}\right).$ Since $\tau$ must satisfy  $$\tau \geq C_1\sqrt{\lambda},$$
we choose
\begin{equation} \label{tau}
\tau = -\frac{1}{B}\ln\left(\frac{1}{2C\sqrt{\lambda}}\frac{\|V\|_R}{\|V\|_{2R}}\right)+C_1\sqrt{\lambda}.
\end{equation}

\noindent Inequality (\ref{3c2}) becomes 

\begin{equation*}
\|V\|_R\leq C\sqrt{\lambda} e^{C_1\sqrt{\lambda}}e^{\frac{-A}{B}\ln\left(\frac{1}{2C_\lambda}\frac{\|V\|_{R}}{\|V\|_{2R}}\right)}\|V\|_{\frac{R}{2}},
\end{equation*}

% \begin{equation*}
 %\|V\|^{\frac{A+B}{B}}_R\leq e^{C_5\sqrt{\lambda}}\|V\|_{2R}^{\frac{A}{B}}\|V\|_{\frac{R}{2}},
% \end{equation*}

 \begin{equation*}
 \|V\|_R\leq  e^{\left(C_1\sqrt{\lambda}\right)\frac{B}{A+B}}\|V\|_{2R}^{\frac{A}{A+B}}\|V\|_{\frac{R}{2}}^{\frac{B}{B+A}}. \end{equation*}
 
\noindent  Finally define  $\alpha=\frac{A}{A+B}$ and replace $C_i$ by $C_i\frac{B}{A+B}$ to have
 \begin{equation*}
 \|V\|_R\leq e^{C_5\sqrt{\lambda}}\|V\|_{2R}^{\alpha}\|V\|^{1-\alpha}_{\frac{R}{2}}.
 \end{equation*} 

\end{proof}

From now on we assume that $M$ is compact. Thus we can derive from three sphere theorem above uniform doubling estimates on solutions.

\begin{thm}[doubling estimates]\label{dor}

There exist two non-negative constants  $R_0$, $C_1$  depending only on  $M$ such that : if $u$ is a solution to   \eqref{fp} on $M$
then  $\forall x_0\in M,\forall r>0,$ 
\begin{equation}\label{do}\|\nabla u\|_{B_{2r}(x_0)}\leq e^{C_1\sqrt{\lambda}}\|\nabla u\|_{B_r(x_0)}.
\end{equation}
\end{thm}
\vspace{0,5cm}
\begin{Rq}
Using standard elliptic theory to bound the $L^\infty$  norm of $|V|$ by a multiple of its $L^2$ norm gives for $\delta>0$ : $$\|V\|_{L^\infty(B_\delta(x_0))}\geq (C_1\lambda+C_2)^{\frac{n}{2}}\delta^{-n/2}\|u\|_{2\delta}$$ Then  one can see that the doubling estimate is still true with the $L^\infty$ norm 
\begin{equation}
\|V\|_{L^\infty(B_{2r}(x_0))}\leq e^{C\sqrt{\lambda}}\|V\|_{L^\infty(B_r(x_0))} 
\end{equation}
\end{Rq}
\noindent  To proove the theorem \ref{dor} we need the following 
 \begin{prop}\label{cor1}
$\forall R>0,\  \exists\ C_R>0,\ \forall x_0\in M :$
$$\|\nabla u\|_{B_R(x_0)}\geq e^{-C_R\sqrt{\lambda}}\|\nabla u\|_{L^2(M)} .$$
\end{prop}

\begin{proof}
Let $R>0$ and assume without loss of generality that $R<R_0$ whith $R_0$  such that three spheres theorem (theorem \ref{ttc}) is valid. Up to multiplication by a constant, we can assume that $\|\nabla u\|_{L^2(M)}=1$. We denote by $\bar{x}$ a point in $M$ such that  $\|\nabla u\|_{B_R(\bar{x})}=\sup_{x\in M}  \|\nabla u\|_{B_R(x)}$.
 This implies that one has  $\|\nabla u\|_{B_{R(\bar{x})}}\geq D_R$, where $D_R$ depend only on $M$ and $R$.
One has  from proposition (\ref{ttc})  at an arbitrary point $x$ of $M$ : 

\begin{equation}\label{cop}\|\nabla u\|_{B_{R/2}(x)}\geq e^{-c\sqrt{\lambda}}\|\nabla u\|^{\frac{1}{\alpha}}_{B_R(x)}\end{equation}
Let   $\gamma$ be a geodesic curve  beetween $x$ and $\bar{x}$ and define  $x_0=x,\ x_1,\cdots,x_m=\bar{x}$ such that 
 $x_i\in\gamma$ and
 $B_{\frac{R}{2}}(x_{i+1})\subset B_R(x_i),\ \forall i=1,\cdots,m$. The constant $m$  depends only on $\mathrm{diam}(M)$ and  $R$. Then the properties of $(x_i)_{1\leq i\leq m}$ and inequality \eqref{cop} give for all $i$, $1\leq i\leq m$ :
\begin{equation}
\|\nabla u\|_{B_{R/2}(x_i)}\geq e^{-c_i\sqrt{\lambda}}\|\nabla u\|^{\frac{1}{\alpha}}_{B_{R/2}(x_{i+1})}.
\end{equation}
The result follows by induction and the fact that $\|\nabla u\|_{B_R(\bar{x})}\geq D_R$.

\end{proof}
\begin{cor}\label{cor2}
For all $R>0$, there exists a positive constant $C_R$ depending only  on $M$ and $R$ such that at any point $x_0$ in $M$ one has
\begin{equation*}
\|\nabla u\|_{\frac{R}{4},\frac{R}{8}}\geq e^{-C_R\sqrt{\lambda}}\|\nabla u\|_{L^2(M)}
\end{equation*}
\end{cor}
\begin{proof} 
Let  $R<R_0$ where $R_0$ is such that the three spheres theorem is valid, note that $R_0\leq \mathrm{diam}(M)$. Recall that we  defined locally near a point $x_0$ : $A_{r_1,r_2}:=\{x\in M ; r_1\leq d(x,x_0)\leq r_2)\}$. As $M$ is geodesically complete, there exists a point $x_1$ in  $A_{\frac{R}{8},\frac{R}{4}}$ 
such that $B_{x_1}(\frac{R}{16})\subset A_{\frac{R}{8},\frac{R}{4}}$. From proposition \ref{cor1} one has 
 $\|\nabla u\|_{B_{\frac{R}{16}}(x_1)}\geq e^{-C_R\sqrt{\lambda}}\|\nabla u\|_{L^2(M)}$
 wich gives the result. 
\end{proof}
\begin{proof}[Proof of theorem \ref{dor}]

We proceed like in the proof of three spheres theorem except for the fact that now we want the first ball to become arbitrary small in front of the others.
 Let  $R=\frac{R_0}{4}$ where $R_0$ is such that the three spheres theorems is valid,  let  $\delta$ such that  $0<\delta<2\delta<3\delta<\frac{R}{8}<\frac{R}{2}<R$,
and  define  a smooth radial function $\psi$, with $0\leq\psi\leq1$  as follows: 
 
\begin{itemize}
\item[$\bullet$] $\psi(x)=0$ if $r(x)<\delta$ or if $r(x)>R$,
\item[$\bullet$] $\psi(x)=1$ if $\frac{5\delta}{4}<r(x)<\frac{R}{2}$,
\item[$\bullet$] $|\nabla\psi(x)|\leq\frac{C}{\delta}$ if $r(x)\in[\delta,\frac{5\delta}{4}]$ and  $|\nabla\psi(x)|\leq C$ if $r(x)\in[\frac{R}{2},R]$,
%\item[$\bullet$]
\item[$\bullet$]  $|\nabla^2\psi(x)|\leq\frac{C}{\delta^2}$ if $r(x)\in[\delta,\frac{5\delta}{4}]$ and $|\nabla^2\psi(x)|\leq C$ if $r(x)\in[\frac{R}{2},R]$.
\end{itemize}
Keeping appropriates terms in \eqref{cutof} gives :  

\begin{equation*}
\begin{array}{rcl}
\|r^{\frac{\varepsilon}{2}}e^{\tau\phi}\psi V\|+ \tau\delta\|r^{-1}e^{\tau\phi}\psi V\|&\leq &C\left(\|r^2e^{\tau\phi}\nabla V\cdot\nabla\psi\|+\|r^2e^{\tau\phi}\Delta\psi V\|\right) \vspace{0,15cm}\\
&\leq &\frac{C}{\delta}\|r^2e^{\tau\phi}\nabla V\|_{\delta,\frac{5\delta}{4}}+C\|e^{\tau\phi}\nabla V\|_{\frac{R}{2},R}\vspace{0,15cm}\\
&+&\frac{C}{\delta^2}\|r^2e^{\tau\phi}V\|_{\delta,\frac{5\delta}{4}}+C\|e^{\tau\phi} V\|_{\frac{R}{2},R}
\end{array}
\end{equation*}
Using properties of $\psi$ we have,  

\begin{equation*}
\begin{array}{rcl}
&&\|r^{\frac{\varepsilon}{2}}e^{\tau\phi}V\|_{\frac{5\delta}{4},3\delta}+\|r^{\frac{\varepsilon}{2}}e^{\tau\phi}V\|_{\frac{R}{8},\frac{R}{4}}\vspace{0,15cm}\\ &+&\tau\delta\|r^{-1}e^{\tau\phi}V\|_{\frac{5\delta}{4},3\delta}+\tau\delta\|r^{-1}e^{\tau\phi}V\|_{\frac{R}{8},\frac{R}{4}}\vspace{0,15cm}\\
&\leq& \frac{C}{\delta}\|r^2e^{\tau\phi}\nabla V\|_{\delta,\frac{5\delta}{4}}+C\|e^{\tau\phi}\nabla V\|_{\frac{R}{2},R}\vspace{0,15cm}\\&+&\frac{C}{\delta^2}\|r^2e^{\tau\phi} V\|_{\delta,\frac{5\delta}{4}}+C\|e^{\tau\phi}V\|_{\frac{R}{2},R}.
\end{array}
\end{equation*}

\noindent Now drop the first and last terms of the left hand side gives :

\begin{eqnarray*}
\|r^{\frac{\varepsilon}{2}}e^{\tau\phi} V\|_{\frac{R}{8},\frac{R}{4}}+\|e^{\tau\phi}V\|_{\frac{5\delta}{4},3\delta} \!\!&\leq  &\!\!C \left(\delta \|e^{\tau\phi}\nabla V\|_{\delta,\frac{5\delta}{4}}+\|e^{\tau\phi}\nabla V\|_{\frac{R}{2},R}\right)\\
&\!\!+\!\!&C\left(\|e^{\tau\phi} V\|_{\delta,\frac{5\delta}{4}}+\|e^{\tau\phi}V\|_{\frac{R}{2},R}\right)\nonumber
\end{eqnarray*}
% On  obtient alors  pour la norme $f\mapsto(\int|f|^2dv_g)^\frac{1}{2}$ notÈe encore $\|\cdot\|$:
% \begin{eqnarray}
% \delta^{\frac{n}{2}}\|e^{\tau\phi} u\|_{\frac{R}{8},\frac{R}{4}}+\|e^{\tau\phi}u\|_{\frac{5\delta}{4},3\delta} \!\!&\leq  &\!\!C \left(\delta \|e^{\tau\phi}\nabla u\|_{\delta,\frac{5\delta}{4}}+\delta^{\frac{n}{2}}\|e^{\tau\phi}\nabla u\|_{\frac{R}{2},R}\right)\\
% &\!\!+\!\!&C\left(\|e^{\tau\phi} u\|_{\delta,\frac{5\delta}{4}}+\delta^{\frac{n}{2}}\|e^{\tau\phi}u\|_{\frac{R}{2},R}\right)\nonumber
% \end{eqnarray}

%\noindent Define  $C_\lambda=C_3\sqrt{\lambda}$  where $C_3$ depends only on $M$ and $R_0$ and for which the value could change from a line to another.
 \noindent Now using  (\ref{lm1}) and properties of   $\phi$,

\begin{eqnarray*}
\|e^{\tau\phi} V\|_{\frac{R}{8},\frac{R}{4}}+\|e^{\tau\phi}V\|_{\frac{5\delta}{4},3\delta}&
\leq& C\sqrt{\lambda}\left(e^{\tau\phi(\delta)}\|\ V\|_{\frac{2\delta}{3},\frac{3\delta}{2}}+e^{\tau\phi(\frac{R}{5})}\|V\|_{\frac{R}{5},\frac{5R}{3}}\right)\nonumber \\
&+&C\sqrt{\lambda}\left(e^{\tau\phi(\delta)}\|V\|_{\delta,\frac{5\delta}{4}}+e^{\tau\phi(\frac{R}{5})}\|V\|_{\frac{R}{2},R}\right)
\end{eqnarray*}

%On peut alors absorber les derniers termes du membre de droite et majorer la norme sur un anneau par celle sur le disque complet 

\begin{equation*}
\|e^{\tau\phi} V\|_{\frac{R}{8},\frac{R}{4}}+\|e^{\tau\phi}V\|_{\frac{5\delta}{4},3\delta} \leq C\sqrt{\lambda}\left(e^{\tau\phi(\delta)}\|V\|_{\frac{3\delta}{2}}+e^{\tau\phi(\frac{R}{5})}\|V\|_{\frac{5R}{3}}\right)
\end{equation*}

\begin{equation*}
e^{\tau\phi(\frac{R}{4})} \|V\|_{\frac{R}{8},\frac{R}{4}}+e^{\tau\phi(3\delta)}\|V\|_{\frac{5\delta}{4},3\delta} \leq C\sqrt{\lambda}\left(e^{\tau\phi(\delta)}\|V\|_{\frac{3\delta}{2}}+e^{\tau\phi(\frac{R}{5})}\|V\|_{\frac{5R}{3}}\right)
\end{equation*}

\noindent Adding $e^{\tau\phi(3\delta)}\|V\|_{\frac{5\delta}{4}}$ to each side  % peut pour $\tau$ suffisamment grand l'absorber par le premier terme dans le membre de droite :
\begin{equation*}
e^{\tau\phi(\frac{R}{4})} \|V\|_{\frac{R}{8},\frac{R}{4}}+e^{\tau\phi(3\delta)}\|V\|_{3\delta} \leq C\sqrt{\lambda}\left(e^{\tau\phi(\delta)}\|V\|_{\frac{3\delta}{2}}+e^{\tau\phi(\frac{R}{5})}\|V\|_{\frac{5R}{3}}\right)
\end{equation*}
%On choisit alors comme dans le thÈorËme des trois cercles  $\tau$ assez grand de tel sorte qu'on ait 
Now we want to choose $\tau$ such that  
$$C\sqrt{\lambda} e^{\tau\phi(\frac{R}{5})}\|V\|_{\frac{5R}{3}}\leq \frac{1}{2}e^{\tau\phi(\frac{R}{4})} \|V\|_{\frac{R}{8},\frac{R}{4}}$$
%On choisit donc : 
For the same reasons than before we choose 
  $$\tau=\frac{1}{\phi(\frac{R}{5})-\phi(\frac{R}{4})}\mathrm{ln}\left(\frac{1}{2C\sqrt{\lambda}}\frac{\|u\|_{\frac{R}{8},\frac{R}{4}}}{\|u\|_{\frac{5R}{3}}}\right)+C_1\sqrt{\lambda}$$
 Define $A=\left(\phi(\frac{R}{5})-\phi(\frac{R}{4})\right)^{-1}$; like before one can assume that $A$ is non-positive and independent of $R$. So,
  $$e^{\tau\phi(\frac{R}{4})}\|V\|_{\frac{R}{8},\frac{R}{4}}+e^{\tau\phi(3\delta)}\|V\|_{3\delta}\leq C\sqrt{\lambda}e^{\tau\phi(\delta)}\|V\|_{\frac{5\delta}{2}}$$
\noindent One can then  ignore the first term of the right hand side to get :%($\phi(3\delta)\geq 0$)
  $$e^{\tau\phi(3\delta)}\|V\|_{3\delta}\leq C\sqrt{\lambda}\  e^{A\mathrm{ln}\left(\frac{1}{2C\sqrt{\lambda}}\frac{\|V\|_{\frac{R}{8},\frac{R}{4}}}{\|V\|_{\frac{5R}{3}}}\right)+C_1\sqrt{\lambda}}\|V\|_{\frac{3\delta}{2}} $$
  $$\|V\|_{3\delta}\leq e^{C\sqrt{\lambda}}\left(\frac{\|V\|_{\frac{R}{8},\frac{R}{4}}}{\|V\|_{\frac{5R}{3}}}\right)^{A}\|V\|_{\frac{3\delta}{2}} $$
  Finally from corollary \ref{cor2}, define  $r=\frac{3\delta}{2}$ to have : 
  $$\|V\|_{2r}\leq e^{C\sqrt{\lambda}}\|V\|_{r} $$
   Thus, the theorem is proved for all $r\leq\frac{R_0}{16}$. Using proposition \ref{cor1} we have for $r\geq \frac{R_0}{16}$ : $$\|\nabla u\|_{B_{x_0}(r)}\geq\|\nabla u\|_{B_{x_0}(\frac{R_0}{16})}\geq e^{-C_0\sqrt{\lambda}}\|\nabla u\|_{L^2(M)}\geq e^{-C_1\sqrt{\lambda}} \|\nabla u\|_{B_{x_0}(2r)}$$
  
\end{proof}

%\noindent Then It follows from simple fact that the vanishing order of  gradient of eigen :
%\begin{thm}\label{ord}Let  $M$  be a compact, connected  manifold . There exists non-negatives constants $C_1$ and $C_2$ which depend only on M  such that
%for any solution $u$ to \eqref{fp} the vanishing order of $u$ at any point of  $M$ is  bounded by $C_1\sqrt{\lambda}+C_2$. 
%\end{thm}
%The square root is the best exponent possible on $W$ as one can see by considering  the sperical harmonic $\mathrm{Re}(x+iy)^k$.
%La racine carrÈe dans la majoration est le meilleur exposant possible pour $W$. }n peut le voir en considÈrant certaines harmoniques sphÈriques ( cf chapitre 4).\newline  
\section{Critical set on analytic manifold}
% Now we will establish a Carleman estimate for the operator  $\Delta+\lambda $ acting on vector functions, which will be useful when dealing with the eigenfunctions of the Laplacian. 
% For $U\in\mathcal{C}^{\infty}_0(B_{R_0}(x_0)\setminus\{x_0\},\mathbb{R}^m)$, applying (\ref{S})  to each components $U^i$ of $U$ and summing gives : 
% \begin{cor}\label{carv}%Under the same assumptions as in theorem \ref{tics} one has, 
% There exist non-negative constants $R_0, C,C_1$, which depend only on  $M$ and $f$, such that : \newline  
% $\forall x_0\in M,\:\forall\: U\:\in\mathcal{C}^{\infty}_0(B_{R_0}(x_0)\setminus\{x_0\},\mathbb{R}^m),\ \forall\  \tau\geq C_1\sqrt{\lambda}, $
% % $\forall W\in \Cc^1(M),,\forall u\in C^\infty_0(B_{R_0}(x_0)\setminus\{0\}), \forall \tau \geq C_1\sqrt{\|W\|_{\mathcal{C}^1}}+C_2, $   
%  
% \begin{eqnarray}C\left\|r^2e^{-\tau\phi}\left(\Delta U +\lambda U \right)\right\|&\geq& \tau^\frac{3}{2}\left\|\sqrt{|f^{\prime\prime}(\ln r)|}e^{-\tau\phi}U\right\| \nonumber\\
% & +& \tau^{\frac{1}{2}}\left\|r\sqrt{|f^{\prime\prime}(\ln r)|}e^{-\tau\phi}\nabla U\right\|
% \end{eqnarray}
% Moreover, 
% $$\mathrm{If}\:\:\:\mathrm{supp}(U)\subset\{x\in M; r(x)\geq\delta>0\},$$ then
% \begin{eqnarray}\label{Sivv2}
% C\left\|r^2e^{-\tau\phi}\left(\Delta U+\lambda U\right)\right\|&\geq& \tau^\frac{3}{2}\left\|\sqrt{|f^{\prime\prime}(\ln r)|}e^{-\tau\phi}U\right\| \nonumber\\
% +\ \tau\delta\left\|r^{-1}e^{-\tau\phi}U\right\|&+& \tau^{\frac{1}{2}}\left\|r\sqrt{|f^{\prime\prime}(\ln r)|}e^{-\tau\phi}\nabla U\right\|.
% \end{eqnarray}
% \end{cor}
From here we will follow the method of Donnelly and Fefferman \cite{DF1} to establish upper bound for the $(n-1)$-dimensionnal
 measure of critical set of eigenfunctions. So we also suppose that $M$ is analytic. 
Recall that $\Nc_u=\left\{x\in M : u(x)=0\right\}$ and $\Cc_u=\left\{x\in M : \nabla u(x)=0\right\}$. Define $B_{\C}(r)$ the complex ball : $$B_{\C}(r)=\left\{z\in\C^n: |z|<r\right\} $$and $B(r)$ the standard ball in $\R^n$ centred at $0$ of radius $r$.
 The main point to deduce from our doubling inequality an estimate on the Hausdorff measure of the  critical set is the following result of Donnelly and Fefferman :

\begin{thm}[\cite{DF1} p. 180]\label{df}
 Let $F$ be an holomorphic function on $B_\C(1)$ and suppose there exists  $\alpha>1$ such that  
$$\max_{B_\C(1)}|F|\leq e^{\alpha}\max_{B(\frac{1}{2})}{|F|},$$ then 
$$\Hc^{n-1}\left(\Nc_F\cap B\left(\frac{1}{4}\right)\right)\leq C\alpha.$$
where $\Nc_F$ is the zero set of $F$ in $\R^n$ and $C$ a constant depending only on the dimension.
\end{thm}
\non Let $u$ be a solution to \eqref{fp}. Fix $x_0$ in $M$ and consider $(x_1,\cdots,x_n)$  a chart around $x_0$. We assume that the chart contains the euclidean  ball $B_2$. We  define 
$$F(x)=\sum_{i=1}^n\left|\frac{\partial u}{\partial x_i}\right|^2,$$  The nodal set of $F$ is the critical set of $u$. 
%On va chercher ‡ appliquer au prolongement analytique de $F$ le thÈorËme \ref{df}. On supposera sans perte de gÈnÈralitÈ que notre carte locale est la boule unitÈ centrÈe en 0. \newline
One has :
\begin{prop}\label{cri} The function $F$  can be extended to an analytic function on $B_\C(1)$ and :  
\begin{equation*}\|F\|_{L^\infty(B_{\C}(1))}\leq e^{C\sqrt{\lambda}}\|F\|_{L^\infty(B(\frac{1}{2}))}\end{equation*}where $C$ is a constant depending only on $M$.
\end{prop}

%Tout d'abord  on utilisera les idÈes de la dÈmonstration du thÈorËme de rÈgularitÈ analytique pour les opÈrateurs hypoelliptiques (voir par exemple  \cite{hor1} p. 178)  pour obtenir 

\begin{lm}\label{der}Let $u$ be an eigenfunction of the laplace operator on  $B(1)$, for all multi-index $\beta$, with $|\beta|\geq1$ one has :   
\begin{equation}\label{der2}|D^\beta u(0)|\leq \beta ! C^{|\beta|}\sqrt{\lambda}^{|\beta |}\|\nabla u\|_{L^\infty\left(B(\frac{C_1}{\sqrt{\lambda}})\right)}
\end{equation}
where $C_1$ is a constant small enough.
\end{lm}
\begin{proof}[proof of lemma \ref{der}]Like in \cite{DF1}, this result can be obtained by rescaling the equation  and using the hypoellipticity proof  (\cite{hor1}, p.178) for an elliptic operator whose coefficients have uniform bounded derivatives.\newline 
Indeed note first that we may assume $\|\nabla u\|_{L^\infty(M)}=1$. Now writing in our local chart $\Delta=\sum_{1\leq |\alpha|\leq2}a_\alpha D^\alpha$ and consider the function $u_{\lambda}(x)=u(\frac{C_1}{\sqrt{\lambda}}x)$, where $C_1$ will be fix below. One can see that $u_\lambda$ is a solution to the elliptic equation $$P_\lambda u_\lambda=u_\lambda$$ with $P_\lambda=\sum_{1\leq|\alpha|\leq2}b_\alpha D^\alpha$ and $$b_\alpha(x)=\frac{\lambda^{-1+\frac{|\alpha|}{2}}}{C_1^{|\alpha|}}a_\alpha\left(\frac{C_1x}{\sqrt{\lambda}}\right).$$
A short computation of $D^\beta b_\alpha$, gives for $C_1$ small enough and any multi-index $\beta$: 
$$\sup_{B_1}|D^\beta b_\alpha(x)|\leq C_2|\beta|!,\:\:\:\:\:   \forall 1\leq|\alpha|\leq2$$
where $C_2$ is a constant depending only on $M$.
Then one can use the hypoellipticity proof (\cite{hor1}) with simple modifications to get for any multi-index $\beta$ with $|\beta|>1$: $$|D^\beta u_\lambda(0)|\leq A^{|\beta|}\beta !.$$ 
\end{proof}

\begin{proof}[Proof of proposition \ref{cri}]

Expanding $V=(\frac{\partial u}{\partial x_1},\cdots,\frac{\partial u}{x_n})$ in its  Taylor series gives $$V(z)=\sum_{|\alpha|\geq 0}\frac{z^\alpha}{\alpha!}D^\alpha V(0),$$
where for $\alpha=(\alpha_1,\cdots,\alpha_n)$ in $\N^n$ and $z=(z_1,\cdots,z_n)$ in $\C^n$ we have set $z^\alpha:=z_1^{\alpha_1}z_2^{\alpha_2}\cdots z_n^{\alpha_n}$ and $\alpha!=\alpha_1!\alpha_2!\cdots\alpha_n!$. Now using \eqref{der2} and summing a geometric series gives for a constant $\rho$ small enough
\begin{equation}\label{comp}\sup_{B_\C(0,\frac{\rho}{\sqrt{\lambda}})}|V(z)| \leq C\sup_{B(0,\frac{C_1}{\sqrt{\lambda}})}|V(x)|. \end{equation} 
Then by translating, in the complex ball $B_\C(1)$, the equation and iterating the estimate \eqref{comp} a multiple of $\sqrt{\lambda}$ times one has 
$$\forall z\in B_{\C}(1),\ |V(z)|\leq C^{\sqrt{\lambda}}\sup_{B(2)}|V(x)|  $$
This implies
\begin{equation}\label{re}
\sup_{B_{\C}(1)}|F(z)|\leq e^{C\sqrt{\lambda}}\sup_{B(2)}|F(x)| 
\end{equation}

\non which gives proposition \ref{cri} by using doubling inequality \eqref{do}.
\end{proof}
%On donne ensuite une inÈgalitÈ de doubling pour $V$, c'est une gÈnÈralisation simple du cas scalaire : 

%\non Finally  using doubling ( theorem \ref{idva})  and the previous proposition    :
% \begin{equation}
%\sup_{B_{\C}(\frac{1}{2})}|F(z)|\leq e^{C_1\sqrt{\lambda}+C_2}\sup_{B(\frac{1}{4})}|F(x)|  
%  \end{equation}
%So the  theorem is proved up to a rescalling.
%\end{proof}
%\non Now we are able to proove our main result : 
%\begin{thm}\label{principal}
%If $u$ is an eigenfunction of the laplace operator then one has
%\begin{equation}
%\Hc^{n-1}(\Nc u)\leq C\sqrt{\lambda}
%\end{equation}

\begin{proof}[proof of theorem \ref{principal}] 
Let $u$ be a solution to \eqref{fp}, let $r_0>0$ a fixed number not larger than the injectivity radius of $M$ and $p$ a arbitrary point in $M$. Let consider a normal chart around $p$. By proposition  \ref{cri} one has that 
$\displaystyle{F=\sum_{i=1..n}\left|\frac{\partial u}{\partial x_i}\right|^2}$  satisfy the hypothesis of theorem \ref{df}.
 Then since the nodal set of $F$ is the critical set of $u$ one has  \begin{equation}\label{nod}\Hc^{n-1}\left(\Cc_u\cap B(p,r_0)\right)\leq C\sqrt{\lambda}\end{equation}
where $C$ depends only on $r_0$ and $M$.\newline
The Theorem \ref{principal} follows by a covering argument since $M$ is compact. 
\end{proof}
\begin{Rq} Since doubling estimates imply vanishing order estimates it follows from lemma 3 of \cite{Bar} that the local  estimate \eqref{nod} 
 is still true on smooth manifold, but without any control on the radius $r_0$. 
\end{Rq}

\bibliographystyle{siam}

\bibliography{quantitativeuniqueness}
\end{document}